\pdfoutput=1\relax
\documentclass[english]{article}
\usepackage{meta/preamble}
\title{{Derived Stone Embedding}}
\date{\today}
\author{Amos Kaminski\thanks{Department of Mathematics, Weizmann Institute of Science \\ amos.kaminski@weizmann.ac.il}}
\begin{document}

\setlength{\abovedisplayskip}{3pt}
\setlength{\belowdisplayskip}{3pt}
\setlength{\abovedisplayshortskip}{0pt}
\setlength{\belowdisplayshortskip}{0pt}

\maketitle
\begin{center}
\end{center}
\begin{abstract}
A classical result, the Stone embedding, characterizes profinite sets as totally disconnected, compact Hausdorff spaces. Building on \cite{pyknotic}, which introduced a derived Stone embedding of the pro-category of $\pi$-finite spaces into pyknotic spaces, this paper uses the $\infty$-topoi machinery to partially characterize the essential image of this embedding, extending the classical characterization to the derived setting.

\end{abstract}

\setcounter{tocdepth}{2}
$\\$
\begin{figure}[H]
  \centering{}
  \setlength{\fboxsep}{-5pt}
  \setlength{\fboxsep}{5pt}
  \frame{\includegraphics[scale=0.14]{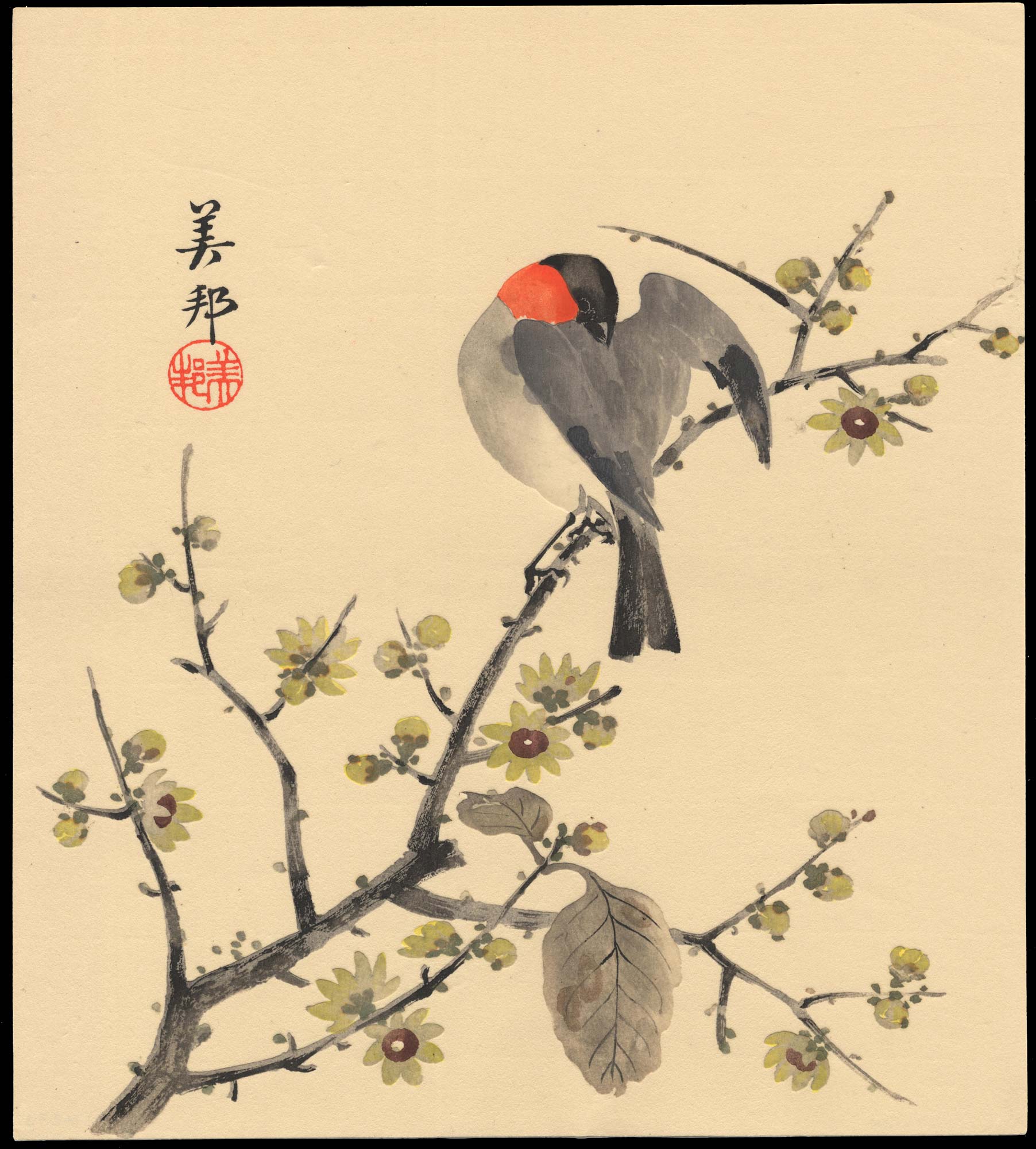}}
  \caption{\footnotesize 
"Bird On Branch (1)" by Takahashi, Biho\\}
    
\end{figure}
\newpage
\begingroup
\renewcommand{\baselinestretch}{1.2}\selectfont
\setlength{\parskip}{6pt}
\tableofcontents
\endgroup
\newpage

\section{Introduction}
\subsection*{Profinite Homotopy Theory}
Profinite methods, exemplified by the \emph{arithmetic square}, allow one to reconstruct a space from its profinite and rational components. For instance:

\begin{theorem}[{\cite[Theorem 7]{lurie_mit}}]
Let $X$ be a simply connected space with finitely generated homotopy groups. Then $X \simeq X_{\mathbb{Q}} \times_{\left(\prod_p \widehat{X}_p\right)_{\mathbb{Q}}} \prod_p \widehat{X}_p.
$
\end{theorem}

The profinite part is naturally studied in the pro-category $\operatorname{Pro}(S_\pi)$, which lacks the structure of an $\infty$-topos making it hard to work with. To address this issue, we would like to embed $\operatorname{Pro}(S_\pi)$ into an $\infty$-topos.
In {\cite[Example 3.3.10]{pyknotic}}(and independently by Scholze and Clausen in \cite{condensed} as condensed objects), pyknotic spaces denoted $\mathbf{Pyk}(\mathcal{S})$, were introduced as hypercomplete sheaves over profinite sets.
It was shown in {\cite[Example 3.3.10]{pyknotic}} that $\operatorname{Pro}(S_\pi)$ can be embedded into $\mathbf{Pyk}(\mathcal{S})$. In fact, $\mathbf{Pyk}(\mathcal{S})$ is in some sense the minimal $\infty$-topos containing $\operatorname{Pro}(S_\pi)$, serving as its $\infty$-topos completion. In this text, we will further analyze this embedding, thereby enhancing this approach and providing a good framework for profinite homotopy theory.

\subsection*{Stone Embedding}
A classical result, known as the \textbf{Stone Embedding}, establishes an embedding of the pro-category of finite sets into the category of topological spaces and characterizes the essential image:
\begin{theorem}[{{\cite[Theorem 2.3 - Page 236]{stonem}}}]
There is an equivalence of categories between the category $\mathbf{Stone}$ of Stone spaces (compact, totally disconnected, Hausdorff spaces) and the pro-category $\mathbf{Pro}(\mathbf{Fin})$ of finite sets.
\end{theorem}
\emph{Pyknotic spaces} provide a natural setting for studying these topological phenomena within an $\infty$-categorical framework.
Building upon this foundation, our paper aims to extend the Stone embedding to the $\infty$-categorical setting. A natural extension of profinite sets to the derived setting is the pro-category of $\pi$-finite spaces. As observed in \cite{pyknotic}, there exists a derived Stone embedding:

\begin{theorem}[{{\cite[Example 3.3.10]{pyknotic}}}]
There is a canonical fully faithful embedding
\[
\operatorname{Pro}(S_\pi) \hookrightarrow \mathbf{Pyk}(\mathcal{S}),
\]
of the pro-category of $\pi$-finite spaces into the category $\mathbf{Pyk}(\mathcal{S})$ of pyknotic spaces.
\end{theorem}

\subsection*{Main Result}
In this paper, we will partially characterize the essential image of the embedding:
\[
\operatorname{Pro}(S_\pi) \hookrightarrow \mathbf{Pyk}(\mathcal{S}).
\]
To achieve this goal, we will apply some tools from the theory of $\infty$-topoi and explore relevant definitions and theorems concerning pyknotic spaces.
In a Mathoverflow post, Peter Scholze claimed a related result (see \cite{scholze}), although without providing a proof. In this paper we proved a partial characterization:
\begin{theorem}[Partial Characterization of the Essential Image -- Theorem 4.6]
\label{cor:essential_image}
Consider the embedding $\operatorname{Pro}(S_\pi) \hookrightarrow \mathbf{Pyk}(\mathcal{S})$. A connected object in $\mathbf{Pyk}(\mathcal{S})$ lies in the essential image of this embedding if and only if its pyknotic homotopy groups at any chosen basepoint are profinite.
\end{theorem}

To achieve this characterization, two key points need to be addressed. First, we demonstrate that the inverse limit commutes with homotopy groups for pro-systems of $\pi$-finite spaces. This is formalized in the following corollary:

\begin{corollary}[Corollary 3.12]
\label{cor:homotopy_limits}
Let \(\{X_i\}\) be a pro-system of \(\pi\)-finite spaces. Denote by \(p_i: \varprojlim X_i \rightarrow X_i\) the projection maps. Then, we have the following isomorphism:
\[
\pi_j(\varprojlim X_i, x_0) = \varprojlim \pi_j(X_i, p_i(x_0))
\]
\end{corollary}

This result implies that the homotopy groups of objects in $\operatorname{Pro}(S_\pi)$ are profinite, providing the first part of the characterization in Corollary~\ref{cor:essential_image}.

Second, to establish the converse, we proceed by induction on the Postnikov tower of a space with profinite homotopy groups. For this inductive process, it was crucial to verify that certain relative profinite Eilenberg--MacLane spaces are indeed in $\operatorname{Pro}(S_\pi)$, ensuring that the building blocks of our induction are profinite. This is established in the following lemma:

For a category $\mathcal{C}$,  $\mathcal{C}/X$ denotes the over category over $X$, and for an $\infty$-topos $\mathfrak{X}$ and a group object $G$ in $\text{Disc}(\mathfrak{X}/X)$,  $K^{/X}(G, n)\in \mathfrak{X}/X$  denotes the Eilenberg MacLane object banded by $G$.
\begin{lemma}[Lemma 4.9]
\label{lem:EM_Profinite}
Let $\Gamma$ be a $1$-gerb in $\textbf{Pyk}(\mathcal{S})$, let $\zeta: 1 \rightarrow \Gamma$ be a point of $\Gamma$, and let $ K^{/\Gamma}(G, n) $ be an Eilenberg--MacLane object in $\mathbf{Pyk}(\mathcal{S})/\Gamma$. Then, $K^{/\Gamma}(G, n) \in \operatorname{Pro}(S_\pi)$ (viewing $K^{/\Gamma}(G, n)$ as an object of $\mathbf{Pyk}(\mathcal{S})$ via the forgetful functor) if and only if both $\Gamma$ and \( G_0 := G \times_\Gamma \zeta \) lie in $\operatorname{Pro}(S_\pi)$.
\end{lemma}

Knowing that these Eilenberg--MacLane objects are profinite, we need to understand the gluing process in our inductive construction. For this purpose, we introduce a well-known crucial lemma(essentially it is {{\cite[Lemma 7.2.2.26]{HTT}}}) that describes the pasting process in an $\infty$-topos:

\begin{lemma}[Lemma 2.15]
\label{lem:pullback_diagram}
Let $X \in \mathfrak{X}$ be a $1$-connected object in an $\infty$-topos $\mathfrak{X}$. Then, there exists a pullback diagram:
\[
\begin{tikzcd}
    \tau^{\leq n}X \arrow[r] \arrow[d] & \tau^{\leq n-1}X \arrow[d] \\
    * \arrow[r] & K\left( \pi_n(X), n+1 \right)
\end{tikzcd}
\]
\end{lemma}

This lemma allows us to reconstruct the Postnikov tower of $X$ by successively attaching Eilenberg--MacLane spaces corresponding to its homotopy groups. By ensuring that each stage of the tower remains within $\operatorname{Pro}(S_\pi)$, we complete the induction and thereby establish the characterization in Corollary~\ref{cor:essential_image}.
\subsection*{Acknowledgments}
First and foremost, I would like to express my heartfelt gratitude to Shachar Carmeli and Dmitry Gourevitch for their invaluable guidance and support as my MSc supervisors.  I am also deeply thankful to Tomer Schlank for his guidance when I was supervised by him at the Hebrew University and for suggesting the initial research idea during my time at the Hebrew University.

Additionally, I wish to extend my appreciation to Lior Yanovski, Yaakov Varshavsky, and the members of Tomer Schlank's research group for their insightful discussions and contributions.

Finally, I also wish to acknowledge my usage of OpenAI's ChatGPT, which aided me with the phrasing and \LaTeX{} formatting of this text.

\section{\texorpdfstring{\ensuremath{\infty}-topoi}{Infinity-topoi}}

In this section, we will review the \(\infty\)-topos machinery relevant to our work. We assume the reader is familiar with the basic definitions of \(\infty\)-topoi, as presented in {{\cite[Chapter 6]{HTT}}}.
All the definitions and propositions in this section are well-established, and we do not claim any originality. We included them partly for completeness and partly because we could not find appropriate references.
\subsection{ Homotopy Groups}
  \ 
  \\ \  
In this section, we will explore the notion of homotopy groups in the context of an $(\infty,1)$-topos. In classical homotopy theory, homotopy groups at a base point are defined as the homotopy classes of pointed maps from the $n$-sphere $S^n$ to a given space.
\\

This concept naturally extends to the more general setting of an $(\infty,1)$-topos. One key fact is that every $(\infty,1)$-topos $\mathfrak{X}$ is "powered over spaces." This means that for any object $X \in \mathfrak{X}$ and any space $K \in \mathcal{S}$, we can construct the power object $X^K$. The power object $X^K$ is defined as the limit of the constant diagram with value $X$ indexed by $K$, that is, $X^K = \varprojlim_K X$. In particular, for the $n$-sphere $S^n$, we can form the power object $X^{S^n}$.
\\

We will also highlight several key differences between the classical and topos-theoretic settings. In an $(\infty,1)$-topos, the homotopy groups $\pi_n(X)$ are not defined globally as group objects, but rather as sheaves over the object $X$. Furthermore, we will introduce \emph{Eilenberg--MacLane objects}, which, as in the classical theory, form the building blocks of the Postnikov tower of an object. Finally, we will discuss the general machinery for reconstructing objects via iterative pullbacks of Eilenberg--MacLane objects(associated with their homotopy groups) along with Postnikov invariants.

\subsubsection{\texorpdfstring{Powering over \ensuremath{\infty}-groupoids}{Powering over Infinity-groupoids}}

\begin{definition}[Powering over \(\infty\)-groupoids]
Let \(\mathfrak{X}\) be an \(\infty\)-topos. Then \(\mathfrak{X}\) is \emph{powered} over \(\mathcal{S}\), the \(\infty\)-category of spaces, in the following sense: for all \(K \in \mathcal{S}\) and \(X \in \mathfrak{X}\), we can define \(X^K\) as the limit over \(K\) of the constant diagram taking the value \(X\). That is,
\[
X^K := \varprojlim_K X.
\]
\end{definition}

With the concept of powering established, we turn our attention to the final geometric morphism, which serves as a crucial link between an \(\infty\)-topos and the \(\infty\)-category of spaces.

\begin{definition}[Final Geometric Morphism]
Let \(\mathfrak{X}\) be an \(\infty\)-topos. The \emph{final geometric morphism}, commonly known as the \emph{global section functor}, is the unique geometric morphism
\[
\Gamma: \mathfrak{X} \rightarrow \mathcal{S}.
\]
\end{definition}

With the final geometric morphism defined, we can define constant stacks, as those objects that belong to the essential image of the pullback functor.

\begin{definition}[Locally Constant Stacks]
Consider the final geometric morphism from an \((\infty,1)\)-topos \(\mathfrak{X}\) to the \(\infty\)-category of spaces \(\mathcal{S}\):
\[
\Gamma: \mathfrak{X} \to \mathcal{S}.
\]
And consider the pullback functor of this morphism denoted by
\[
\Gamma^*: \mathcal{S} \to \mathfrak{X}.
\]
An object in the essential image of the functor \(\Gamma^*\) is called a \emph{constant stack}.
\end{definition}

We can rephrase the powering operations using constant objects as follows
(We will denote by $\text{Map}$ the internal mapping object in $\mathfrak{X}$.)
\begin{proposition}
\label{link}
For any $X \in \mathfrak{X}$ and any space $K \in \mathcal{S}$, we have the equivalence:
\[
X^K \simeq \text{Map}(\Gamma^*(K), X),
\]
\end{proposition}
\begin{proof}
First, consider the hom-space $\text{Hom}(Z, \text{Map}(\Gamma^*(K), X))$. By the adjunction $$(-) \times Y \dashv \text{Map}(Y, -),$$ we have the equivalence:
\[
\text{Hom}(Z, \text{Map}(\Gamma^*(K), X)) \simeq \text{Hom}(Z \times \Gamma^*(K), X).
\]
Next, switching $Z$ and $\Gamma^*(K)$, and using the same adjunction again, we obtain:
\[
\text{Hom}(Z \times \Gamma^*(K), X) \simeq \text{Hom}(\Gamma^*(K), \text{Map}(Z, X)).
\]
Now, applying the adjunction $\Gamma^* \dashv \Gamma$ :
\[
\text{Hom}(\Gamma^*(K), \text{Map}(Z, X)) \simeq \text{Hom}(K, \Gamma \text{Map}(Z, X)) \simeq \text{Hom}(K, \text{Hom}(*, \text{Map}(Z, X))).
\]
By using the adjunction $(-) \times Y \dashv \text{Map}(Y, -)$ once more, we deduce that:
\[
\text{Hom}(*, \text{Map}(Z, X)) \simeq \text{Hom}(* \times Z, X) \simeq \text{Hom}(Z, X).
\]
Thus:
\[
\text{Hom}(K, \text{Hom}(*, \text{Map}(Z, X))) \simeq \text{Hom}(K, \text{Hom}(Z, X)).
\]
Now, observe that $\text{Hom}(K, \text{Hom}(Z, X))$ is equivalently $\varprojlim_K \text{Hom}(Z, X)$, and since the hom-functor commutes with limits, we conclude that:
\[
\varprojlim_K \text{Hom}(Z, X) \simeq \text{Hom}(Z, \varprojlim_K X) = \text{Hom}(Z, X^K).
\]
Since all the isomorphisms above are natural, by the Yoneda lemma, we obtain the desired equivalence:
\[
X^K \simeq \text{Map}(\Gamma^*(K), X).
\]
\end{proof}
\subsubsection{Homotopy Groups}
With the notion of powering an object over a space, we turn to define homotopy groups of objects in a general infinity topos.
\begin{definition}

Let $\mathfrak{X}$ be an $(\infty, 1)$-topos, and let $X \in \mathfrak{X}$ be an object. We define the \emph{$n$-th homotopy group} (or \emph{homotopy sheaf}) of $X$, denoted by $\pi_n(X)$, as an object in the category $\text{Disc}(\mathfrak{X}/X)$ of discrete objects over $X$ by:
\[
    \pi_n(X) = \tau^{\leq 0}\left(X^{S^n} \to X\right)
\]
where $X^{S^n}$ is the power object, $\tau^{\leq 0}$ refers to the \emph{truncation functor}, and the map $X^{S^n} \to X$ is an evaluation at a given base point.
\end{definition}

In classical homotopy theory, homotopy groups $\pi_n(X, x_0)$ are defined at a given base point $x_0$ and form global group objects. However, in the context of an $(\infty,1)$-topos, the homotopy groups $\pi_n(X)$ are not global objects but are instead defined as sheaves over the object $X$. While the classical setting ties homotopy groups to a specific base point, the topos-theoretic approach generalizes this concept, allowing homotopy groups to be defined more flexibly.

We will now introduce the notion of homotopy groups at a given point in a general topos.

\begin{definition}
Let $X$ be an object in an $(\infty, 1)$-topos $\mathfrak{X}$. Given a map $* \xrightarrow{\xi} X$ (i.e., a point of $X$), we define the homotopy groups $\pi_k(X, \xi)$ of $X$ at the point $\xi$ as:
\[
\pi_k(X, \xi) = \xi^* \pi_k(X) \in \mathfrak{X}_{/*} \simeq \mathfrak{X}.
\]
Here, $\xi^*$ is the pullback along the geometric morphism $\mathfrak{X}/*\rightarrow \mathfrak{X}/X$.
\end{definition}
The homotopy groups are objects in $\text{Disc}(\mathfrak{X}/X)$. However, the lemma below demonstrates that these homotopy groups can actually be defined as sheaves over $\tau^{\leq 1}X$. Consequently, the homotopy groups of a $1$-connected object are defined in $\text{Disc}(\mathfrak{X})$.
\begin{lemma}
\label{disobj}
Let $\mathfrak{X}$ be an $(\infty,1)$-topos, and let $X \in \mathfrak{X}$. Then the truncations maps induce functors:
\[
\text{Disc}(\mathfrak{X}/X) \to \text{Disc}(\mathfrak{X}/\tau^{\leq n}X) \to \text{Disc}(\mathfrak{X}/\tau^{\leq 1}X)
\]
which are equivalences for all $n \geq 1$.
\end{lemma}
\begin{proof}
This is a direct consequence of {\cite[Lemma 7.2.1.13]{HTT}} 
\end{proof}
Before moving forward, it is essential to introduce a concept that will enable us to generalize Eilenberg–MacLane objects from the category of spaces to the framework of a general topos. This concept is known as an \emph{$n$-gerbe}, defined as follows:

\begin{definition}
An \emph{$n$-gerbe} in an $(\infty,1)$-topos is an object that is both $n$-truncated and $(n-1)$-connected.
\end{definition}

From {\cite[Lemma 7.2.2.8]{HTT}}, we have the equivalence $\mathfrak{X}_{*} \simeq \mathfrak{X}^{*/}$, which identifies the category of pointed objects in $\mathfrak{X}$ denoted $\mathfrak{X}_*$ with the category of objects under the final object $* \in \mathfrak{X}$ denoted $\mathfrak{X}^{*/}$. Therefore, we will abuse notation and treat pointed objects in $\mathfrak{X}$ as objects in $\mathfrak{X}$ under $*$.

\begin{proposition}[{\cite[7.2.2.12.]{HTT}}] \label{EM}
Let $\mathfrak{X}$ be an $\infty$-topos, and let $n \geq 0$. Consider the functor $\pi_n: \mathfrak{X}_{*} \rightarrow \text{Disc}(\mathfrak{X})$.

Then:

\begin{enumerate}
    \item \textbf{For $n = 0$:} The functor $\pi_0$ induces an equivalence between the category of $0$-gerbes in $\mathfrak{X}_{*}$ and the category of pointed discrete objects in $\mathfrak{X}$.
    
    \item \textbf{For $n = 1$:} The functor $\pi_1$ induces an equivalence between the category of $1$-gerbes in $\mathfrak{X}_{*}$ and the category of discrete group objects in $\mathfrak{X}$.
    
    \item \textbf{For $n \geq 2$:} The functor $\pi_n$ induces an equivalence between the category of $n$-gerbes in $\mathfrak{X}_{*}$ and the category of discrete abelian group objects in $\mathfrak{X}$.
\end{enumerate}
\end{proposition}
We now introduce the fundamental construction of Eilenberg–MacLane objects, which serve as the building blocks in the Postnikov tower as we shall see in \ref{pulldia}.
\begin{definition}
We define the functor $A \to  K(A, n)$ as the inverse of the functor $\pi_n$. 
In other words, for any discrete group object (abelian if $n\geq 2$) $A \in \text{Disc}(\mathfrak{X})$, the object $(1 \to K(A, n))$ in $\mathfrak{X}_*$ corresponds to $A$ under the equivalence provided by $\pi_n$.
$K(A, n)$ is known as an Eilenberg--MacLane object of degree $n$ banded over $A$.
\end{definition}
\begin{remark}
From \ref{EM} one can think of an Eilenberg Maclane object as a pointed $n$-gerb.
\end{remark}
In our study of pointed objects within an $\infty$-topos, the notion of the loop space plays an important role since the homotopy groups of the loop object are shifted by 1 as we will see in \ref{loophom}. We will now define the concept of loop object:

\begin{definition}

Let $\mathfrak{X}$ be an $\infty$-topos, and let $(* \to X) \in \mathfrak{X}_*$ be a pointed object. The loop object $\Omega X \in \mathfrak{X}_*$ is defined as the following pullback in $\mathfrak{X}_*$:

\[
\begin{tikzcd}
	\Omega X \arrow[r] \arrow[d] & * \arrow[d] \\
	* \arrow[r] & X
\end{tikzcd}
\]
\end{definition}
\begin{proposition}
\label{loophom}
Let \(\mathfrak{X}\) be an \(\infty\)-topos, and let $(* \xrightarrow{p_0} X) \in \mathfrak{X}_*$ be a pointed object of \(\mathfrak{X}\). Then, \(\pi_k(\Omega X, p_0) = \pi_{k+1}(X, p_0)\).
\end{proposition}
\begin{proof}
First, observe that the functor $(-)^{S^n}$  preserves limits because it is itself a limit. As a result, this functor commutes with the loop space operation(all limits are taken in $ \mathfrak{X}_*$).
Therefore $\Omega( X^{S^n})\simeq ( \Omega X)^{S^n} $  and from \ref{link} we know that $( \Omega X)^{S^n} \simeq Map(\Gamma^*(S^n), \Omega X) $.
Now since the internal mapping object preserves all limits we get that 
\[Map(\Gamma^*(S^n), \Omega X)\simeq Map(\Sigma (\Gamma^*(S^n)),  X)\]

Since $\Gamma^*(S^n)$ is a left adjoint, it commutes with colimits, therefore \[
Map(\Sigma (\Gamma^*(S^n)),  X) \simeq Map( \Gamma^*(\Sigma S^n),  X)
\simeq \]
\[Map( \Gamma^*(S^{n+1}),  X)) \simeq X^{S^{n+1}} \]
Therefore, the homotopy groups of $\Omega X$ are shifted by one degree, meaning that \(\pi_k(\Omega X, p_0) = \pi_{k+1}(X, p_0)\).
\end{proof}

We can now observe that the loop space of an Eilenberg–MacLane space is itself an Eilenberg–MacLane space, but of one degree lower:

\begin{corollary}
\label{loop_EM}
Let $\mathfrak{X}$ be an $\infty$-topos, let $n\geq 1$, and let $A \in \text{Gr}(\text{Disc}(\mathfrak{X}))$ be a discrete group object in $\mathfrak{X}$, which is required to be abelian if $n\geq 2$. Then, the loop space of the Eilenberg--MacLane object $K(A, n)$ is $K(A, n-1)$.
\end{corollary}
\begin{proof}
Applying \ref{EM}, and the previous lemma we see that the loop space is an $(n-1)$-Eilenberg--MacLane object $K(A, n-1)$ banded by $A$.
\end{proof}
The following lemma captures an important property of truncation maps in an $\infty$-topos, showing that these maps naturally form $n$-gerbes:
\begin{lemma}
\label{truncngerb}
Let $\mathfrak{X}$ be an $\infty$-topos, and let $X \in \mathfrak{X}$. Then the map $\tau^{\leq n}X \xrightarrow{f} \tau^{\leq n-1}X$ is an $n$-gerbe. Moreover, $\pi_n(f) \simeq \pi_n(X) \in \text{Disc}(\mathfrak{X}/X) \simeq \text{Disc}(\mathfrak{X}/\tau^{\leq 1}X)$.
\end{lemma}

\begin{proof}
By {\cite[ Lemma 5.5.6.14]{HTT}}, the map $f$ is $n$-truncated.

Consider the exact sequence from {\cite[Remark 6.5.1.5]{HTT}} for the composition:
\[
\tau^{\leq n}X \xrightarrow{f} \tau^{\leq n-1}X \xrightarrow{g} *.
\]
This gives the sequence:
\[
\cdots \to f^* \pi_{n+1}(g) \to \pi_n(f) \to \pi_n(g \circ f) \to f^* \pi_n(g) \to \pi_{n-1}(f) \to \cdots
\]

Note that:
\[
\pi_n(g) = \pi_n(\tau^{\leq n-1}X).
\]
From {\cite[ Lemma 6.5.1.9]{HTT}}, we know for all $k < n$:
\[
\pi_n(g \circ f) \simeq \pi_n(\tau^{\leq n}X) \simeq f^* \pi_k(\tau^{\leq n-1}X) \simeq f^* \pi_k(g).
\]

Thus, using the exact sequence we conclude:
\begin{enumerate}
    \item $\pi_k(f) = *$ for all $k < n$.
    \item $\pi_n(f) = \pi_n(X)$.
\end{enumerate}

Finally, using {\cite[Proposition 7.2.1.14]{HTT}}, we conclude that $f$ is an effective epimorphism, hence \textbf{$f$ is $(n-1)$-connected}.
\end{proof}
Next, we establish a key structural property for 1-connected objects in an $\infty$-topos. This result is analogous to the classical theory of Postnikov invariants for principal fibrations, providing a way to understand the layers of an object's homotopy. 

\begin{lemma}
\label{pulldia}
Let $X \in \mathfrak{X}$ be a 1-connected object in an $\infty$-topos. Then, there exists a pullback diagram:
\[
\begin{tikzcd}
    \tau^{\leq n}X \arrow[r] \arrow[d] & \tau^{\leq n-1}X \arrow[d] \\
    * \arrow[r] & K(\pi_n(X), n+1)
\end{tikzcd}
\]
\end{lemma}

\begin{proof}
From Lemma \ref{truncngerb}, we know that the map $\tau^{\leq n}X \to \tau^{\leq n-1}X$ belongs to the category of $n$-gerbes banded over $\pi_n(X)$, denoted $Gerb^{\pi_n(X)}_n(\mathfrak{X})$(see {\cite[Remark 7.2.2.23]{HTT}}  for a detailed definition of $Gerb^{A}_n(\mathfrak{X})$ ). 
From the proof of \cite[Lemma 7.2.2.26]{HTT}, the map $1 \to K(\pi_n(X), n+1)$ is the final object in $Gerb^{\pi_n(X)}_n(\mathfrak{X})$. Therefore, we obtain the following pullback diagram(Morphisms in $Gerb^{\pi_n(X)}_n(\mathfrak{X})$ are given by pullbacks diagram):
\[
\begin{tikzcd}
    \tau^{\leq n}X \arrow[r] \arrow[d] & \tau^{\leq n-1}X \arrow[d] \\
    1 \arrow[r] & K(\pi_n(X), n+1)
\end{tikzcd}
\]
\end{proof}
Before we move forward, it is important to understand how geometric morphisms interact with Eilenberg--MacLane objects. 

\begin{proposition}
Let $f: \mathfrak{X} \to \mathfrak{Y}$ be a geometric morphism, and let $f^*: \mathfrak{Y} \to \mathfrak{X}$ be its pullback functor. Then,for any integer $n$ and any group object $G$ (abelian if $n\geq 2$)  in $\mathfrak{Y}$ , we have:
\[
f^*(K(G,n)) = K(f^* G, n),
\]
\end{proposition}
\begin{proof}
By {\cite[Remark 6.5.1.4]{HTT}}, $f^*(K(G, n))$ is an $n$-gerbe banded over $f^* G$. Using \ref{EM}, this is precisely the object $K(f^* G, n)$.
\end{proof}

\begin{corollary}
\label{EM-transfer}
Let $G$ be a group object in $\text{Disc}(\mathfrak{X}/X)$, let $K^{/X}(G,n)$ denote an Eilenberg--MacLane object in $\mathfrak{X}/X$, banded over $G$, let $\xi:*\rightarrow X$ a point of $X$, and define $G_0 := G \times_\zeta *$. The following diagram is a pullback diagram:
\[
\begin{tikzcd}
    K(G_0, n) \arrow[r] \arrow[d] & {K^{/X}(G,n)} \arrow[d] \\
    1 \arrow[r, "\zeta"'] & X
\end{tikzcd}
\]

\end{corollary}

\begin{proof}
This result follows directly from the previous proposition for the etale geometric morphism $\mathfrak{X} \to \mathfrak{X}/X$.
\end{proof}

\begin{corollary}
\label{pulltrunc}
Let $\mathfrak{X}$ be an $\infty$-topos, and let $X \in \mathfrak{X}$ with a point $\zeta: * \to X$. Then, we have the following pullback diagram:
\[
\begin{tikzcd}
    K(\pi_n(X, \zeta), n) \arrow[r] \arrow[d] & \tau^{\leq n}X \arrow[d] \\
    {*} \arrow[r, "\tau^{\leq n-1}\zeta"] & \tau^{\leq n-1}X
\end{tikzcd}
\]
\end{corollary}
\begin{proof}
The map $\tau^{\leq n}X \to \tau^{\leq n-1}X$ can be viewed as a morphism in $\mathfrak{X}_*$, equipped with the basepoint induced by $\zeta$.  
By \ref{truncngerb}, this map exhibits $\tau^{\leq n}X \to \tau^{\leq n-1}X$ as an $n$-gerbe banded by 
\[
\pi_n(X) \in \mathrm{Disc}(\mathfrak{X}_*/X) \xrightarrow{\; \zeta^* \;} \mathrm{Disc}(\mathfrak{X}_*).
\]  
Hence, by {\cite[Proposition 6.5.1.16 (6)]{HTT}}, its fiber $F$ is itself an $n$-gerbe in $\mathfrak{X}_*$.  
Finally, applying \ref{EM}, we deduce that $F$ is an Eilenberg–MacLane object, and since from {\cite[Remark 6.5.1.4. ]{HTT}} $\pi_n$ commutes with pullback along geometric morphism $F$ is banded by $\pi_n(X,\zeta)$.
\end{proof}

\subsection{Finiteness Properties}
This subsection provides a brief partial summary of Sections A2 and A7 from \cite{SAG}.
\subsubsection{Coherent Topoi}
We now introduce the concepts of coherent objects and coherent topoi, which will play a crucial role in our study of the embedding of pro-$\pi$-finite spaces into pyknotic spaces. Later, we will show that pro-$\pi$-finite spaces are, in fact, coherent objects and that the global sections functor is conservative on these coherent objects.
To begin, we define the notion of quasi-compactness that corresponds to $0$-coherence.

\begin{definition}[{\cite[Definition A.2.0.12.]{SAG}}]

Let \(\mathfrak{X}\) be an \(\infty\)-topos. We say that \(\mathfrak{X}\) is \emph{quasi-compact} provided that every covering of the terminal object \(* \in \mathfrak{X}\) admits a finite subcover. 
In other words, for any effective epimorphism
\[
\coprod_{\lambda \in \Lambda} V_\lambda \to *,
\]
there exists a finite subset \( J \subseteq \Lambda \) such that the induced morphism
\[
\coprod_{j \in J} V_j \to *
\]
is also an effective epimorphism
Furthermore, an object \(X \in \mathfrak{X}\) is called \emph{quasi-compact} (or \emph{0-coherent}) if the slice \(\infty\)-topos \(\mathfrak{X}/X\) is quasi-compact.
\end{definition}
To further develop the concept of coherence in an $\infty$-topos, we use a notion that builds inductively. We already defined what it means for an $\infty$-topos to be \emph{0-coherent} and now we will extend this idea to higher levels of coherence. The formal definition is as follows:

\begin{definition}[{\cite[Definition A.2.0.12.]{SAG}}]
$  $
\begin{enumerate}[label=(\alph*)]

        \item \textbf{Base Case (\(n = 0\))}: An \(\infty\)-topos \(\mathfrak{X}\) is called \emph{0-coherent} if it is quasi-compact.
        
        \item \textbf{Inductive Step}: Assume the notion of \(n\)-coherent is defined for some \(n \geq 0\).
        \begin{enumerate}[label=(\roman*)]
            \item An object \(U \in \mathfrak{X}\) is \emph{\(n\)-coherent} if the slice \(\infty\)-topos \(\mathfrak{X}_{/U}\) is \(n\)-coherent.
            
            \item The \(\infty\)-topos \(\mathfrak{X}\) is \emph{locally \(n\)-coherent} if every object \(X \in \mathfrak{X}\) admits an effective epimorphism from a coproduct of \(n\)-coherent objects:
            \[
            \coprod_{i} U_i \to X \quad \text{with each } U_i \text{ being } n\text{-coherent}.
            \]
            
            \item \(\mathfrak{X}\) is \emph{\((n+1)\)-coherent} provided it is locally \(n\)-coherent and the \(n\)-coherent objects are closed under finite products.
\end{enumerate}
\end{enumerate}
\end{definition}

\begin{definition}
We say that $\mathfrak{X}$ is coherent if it is $n$-coherent for all $n$, an object $X$ is coherent if $\mathfrak{X}/X$ is coherent 
\end{definition}
Recall that a Grothendieck topology is finitary if every cover has a finite subcover
\begin{theorem}[{{\cite[Proposition A.3.1.3.]{SAG}}}]
\label{A.3.1.3}
Consider a small \( \infty \)-category \( \mathcal{C} \) that has pullbacks and is endowed with a finitary Grothendieck topology, then:
\begin{enumerate}
    \item  The functor \( j: \mathcal{C} \to \operatorname{Shv}(\mathcal{C}) \), obtained by composing the Yoneda embedding \( \mathcal{C} \hookrightarrow \operatorname{PSh}(\mathcal{C}) \) with the sheafification functor \( \operatorname{PSh}(\mathcal{C}) \to \operatorname{Shv}(\mathcal{C}) \), maps every object \( C \in \mathcal{C} \) to a coherent object in \( \operatorname{Shv}(\mathcal{C}) \). 

    \item The \( \infty \)-topos \( \operatorname{Shv}(\mathcal{C}) \) is \textbf{locally coherent} - any object \( X \) in \( \operatorname{Shv}(\mathcal{C}) \), there exists a covering family \( \{ U_i \to X \} \) where each \( U_i \) is a coherent object. 

    \item If the category \( \mathcal{C} \) possesses a final (terminal) object, then the entire sheaf topos \( \operatorname{Shv}(\mathcal{C}) \) is \textbf{coherent}.
\end{enumerate}
\end{theorem}
We will now prove that coherent objects within a coherent topos possess coherent homotopy groups.
\begin{theorem}
\label{cohom}
Let $X$ be a coherent object in a coherent topos $\mathfrak{X}$ then:
\begin{enumerate}
    \item The truncation $\tau^{\leq 0} X$ is coherent in the discrete topos $\text{Disc}(\mathfrak{X})$.
    \item For every point $x_0$ of $X$, the homotopy group $\pi_j(X, x_0)$ is coherent.
\end{enumerate}
\end{theorem}

\begin{proof}

By {\cite[Corollary A.2.4.4]{SAG}}, we know that the truncation $\tau^{\leq n} X$ is coherent for all $n$. Now, consider the pullback diagram from \ref{pulltrunc}:

% https://q.uiver.app/#q=WzAsNCxbMCwwLCJLKFxccGlfe24rMX0oWCwgeF8wKSwgbisxKSJdLFsxLDAsIlxcdGF1XntcXGxlcSBuKzF9WCJdLFsxLDEsIlxcdGF1XntcXGxlcSBufVgiXSxbMCwxLCIqIl0sWzEsMl0sWzMsMiwieF8wIl0sWzAsM10sWzAsMV1d
\[\begin{tikzcd}
	{K(\pi_{n+1}(X, x_0), n+1)} & {\tau^{\leq n+1}X} \\
	{*} & {\tau^{\leq n}X}
	\arrow[from=1-1, to=1-2]
	\arrow[from=1-1, to=2-1]
	\arrow[from=1-2, to=2-2]
	\arrow["{x_0}", from=2-1, to=2-2]
\end{tikzcd}\]

Since coherent objects in a coherent topos are closed under pullbacks (as stated in {\cite[Remark A.2.1.8]{SAG}}), and the point is coherent since $\mathfrak{X}$ is coherent we conclude that $K(\pi_{n+1}(X, x_0), n+1)$ is coherent. 

Now, applying \ref{loop_EM} we see that  $\pi_j(X, x_0)$ is given by iterated loop spaces of Eilenberg-MacLane objects, i.e., $\Omega^{n+1} K(\pi_{n+1}(X, x_0), n+1) \simeq \pi_j(X, x_0)$ therefore $\pi_j(X, x_0)$ is obtained through iterated pullbacks of coherent objects. Each pullback is performed along the point which is coherent (since the topos is coherent), thus from {\cite[Remark A.2.1.8]{SAG}} we conclude that $\pi_j(X, x_0)$ is coherent.
\end{proof}

\subsubsection{Bounded And Localic Topoi}
In this section, we will define the notions of an \( n \)-localic topos and of a bounded topos. These are the fundamental concepts used to define the \textbf{solidification} of a topos, which is essential for the \textbf{derived Stone embedding}.

We begin by defining an \( n \)-topos, which can be thought of as a topos in which all objects are \( (n-1) \)-truncated.
\begin{definition}[{{\cite[Definition 6.4.1.1.]{SAG}}}]
For any non-negative integer \( n \), an \(\infty\)-category \( \mathcal{E} \) is called an \( n \)-\emph{topos} if there exists a small \(\infty\)-category \( \mathcal{A} \) and an accessible left-exact localization functor
\[
L: \operatorname{PShv}_{\leq n-1}(\mathcal{A}) \rightarrow \mathcal{E},
\]
where \( \operatorname{PShv}_{\leq n-1}(\mathcal{A}) \) denotes the full subcategory of the presheaf category \( \operatorname{PShv}(\mathcal{A}) \) consisting of \((n-1)\)-truncated presheaves.
\end{definition}
Let \(\mathcal{X}\) and \(\mathcal{Y}\) be \(\infty\)-topoi. The notation \(\text{Fun}_*(\mathcal{Y}, \mathcal{X})\) denotes the \(\infty\)-category of geometric morphisms from \(\mathcal{Y}\) to \(\mathcal{X}\).
\begin{definition}[{\cite[Definition 6.4.5.8]{HTT}}]
let \(\mathfrak{X}\) be an \(n\)-topos ($n$ can be $\infty$). We define \(\mathfrak{X}\) to be \(k\)-localic if, for any \(n\)-topos \(\mathfrak{Y}\), the canonical map
\[
\text{Fun}_*(\mathfrak{Y}, \mathfrak{X}) \rightarrow \text{Fun}_*(\tau_{\leq k-1} \mathfrak{Y}, \tau_{\leq k-1} \mathfrak{X})
\]
is an equivalence.
\end{definition}
\begin{proposition}
A topos is \(k\)-localic if and only if it is equivalent to the category of sheaves \(\text{Sh}(\mathcal{C})\), where \(\mathcal{C}\) is a small \(k\)-category with finite limits.

\end{proposition}
\begin{proof}
See the proof of \cite[Proposition 6.4.5.9.]{HTT}.
\end{proof}
Now, we present the example of the category of spaces. The category of spaces is equivalently the category of sheaves over the point, which is 0-truncated, and therefore the category of spaces is \(0\)-localic.

\begin{example}[{\cite[Example A.2.1.7.]{SAG}}]
\label{S}
The category $\mathcal{S}$ is coherent and $0$-localic, and its coherent objects are the $\pi$-finite spaces.
\end{example}
In the following definition, we will define the notion of a bounded topos. However, since this text deals only with \(n\)-localic toposes, which are already bounded, the reader can safely interpret "bounded topos" as referring to \(n\)-localic topoi for simplicity.
\begin{definition}
A bounded topos is one that can be expressed as a small inverse limit \(\varprojlim \mathfrak{X}_\alpha\), where each \(\mathfrak{X}_\alpha\) is \(n\)-localic for some \(n\), with \(n\) possibly varying for different \(\alpha\).
\end{definition}
\begin{remark}
For example, an \(n\)-localic topos is always bounded.
\end{remark}
\begin{theorem}[{\cite[Proposition 3.3.8]{pyknotic}}]
\label{procov}
Let $n \geq 1$ be an integer, and let $\mathbf{X}$ be an $n$-localic coherent topos. Then every object $X \in \mathrm{Pro}(\mathbf{X}_{<\infty}^{\mathrm{coh}})$ admits an effective epimorphism $Y \twoheadrightarrow X$, where $Y \in \mathrm{Pro}(\mathbf{X}_{\leq n-1}^{\mathrm{coh}})$.
\end{theorem}

\begin{proof}
By \cite[Theorem A.7.5.3]{SAG}, there is an equivalence
\[
\mathbf{X}\simeq Sh_{eff}(\mathbf{X}_{\leq n-1}^{\mathrm{coh}}).
\]
Thus, every coherent object $X \in \mathbf{X}^{\mathrm{coh}}$ admits an effective epimorphism of the form
\[
\coprod_{i} U_i \twoheadrightarrow X,
\]
with each $U_i \in \mathbf{X}_{\leq n-1}^{\mathrm{coh}}$. Since $X$ is coherent and the category of $(n-1)$-coherent objects is closed under finite coproducts, it follows that this effective epimorphism can be taken from a single object in $\mathbf{X}_{\leq n-1}^{\mathrm{coh}}$.

The extension to pro-objects follows by induction. For further details, see the proof given explicitly in \cite{pyknotic}.
\end{proof}

\vspace*{10mm}

\section{Pyknotic Spaces}
\subsection{Introduction}
Topological spaces exhibit both combinatorial aspects---captured by homotopy types---and analytic aspects that homotopy types alone cannot fully describe. However, the category of topological spaces lacks desirable features such as internal Hom objects and robust integration with algebraic structures. To address these challenges, two complementary frameworks have been introduced: \textbf{homotopy types} for capturing the combinatorial aspects and \textbf{pyknotic sets}(see \cite{pyknotic}) (or, very similarly, \textbf{condensed sets} -- see \cite{condensed}) for the analytic aspects.

Pyknotic sets form a \textbf{Grothendieck topos}, essentially defined as sheaves over the category of profinite sets, with minor adjustments due to set-theoretic concerns. By changing the target category of the sheaves, one can preserve desirable properties if the target category is sufficiently structured. For instance, pyknotic abelian groups form an abelian category, unlike the category of topological abelian groups, which is not abelian other examples includes pyknotic spaces, pyknotic rings, etc. 
This makes pyknotization a suitable framework for encoding analytic data while ensuring compatibility with algebraic and categorical structures.

\subsection{Basic Concepts}
According to {\cite[Proposition A.3.2.1]{SAG}}, one can define a finitary Grothendieck topology on an $\infty$-category $\mathcal{C}$ by specifying that a collection of morphisms $\{U_i \to U\}$ constitutes a covering family if and only if there exists a finite subcollection for which the induced map
\[
\coprod U_i \twoheadrightarrow U
\]
is an effective epimorphism. We refer to this topology as the \emph{effective topology} on $\mathcal{C}$.

\begin{definition}[{\cite[2.2.1]{pyknotic}}]
Consider the following sites:
\[
\mathbf{EStn} \subseteq \mathbf{Stn} \subseteq \mathbf{Comp},
\]
defined as:

\begin{itemize}
    \item $\mathbf{Comp}$: the category of \emph{compact Hausdorff spaces} that are tiny \footnote{That is, $\delta_0$-small, where $\delta_0$ denotes the smallest inaccessible cardinal.}.
    \item $\mathbf{Stn}$: the full subcategory of $\mathbf{Comp}$ consisting of \emph{Stone spaces} namely compact Hausdorff totally disconnected spaces.
    \item $\mathbf{EStn}$: the full subcategory of $\mathbf{Stn}$ consisting of \emph{Stonean spaces}, namely compact Hausdorff extremally disconnected spaces.
\end{itemize}
Each of these categories is equipped with the effective topology.
\end{definition}

It is shown in \cite[2.2.1]{pyknotic} that the restriction functors induce equivalences between the corresponding categories of sheaves:
\[
\mathbf{Sh}_{\mathrm{eff}}(\mathbf{Comp}; \mathbf{Set}) \simeq \mathbf{Sh}_{\mathrm{eff}}(\mathbf{Stn}; \mathbf{Set}) \simeq \mathbf{Sh}_{\mathrm{eff}}(\mathbf{EStn}; \mathbf{Set}).
\]

We now proceed to define the notions of pyknotic sets and pyknotic spaces.

\begin{definition}[Pyknotic Sets, {\cite[Definition 2.13]{pyknotic}}]
The category of \emph{pyknotic sets}, denoted $\mathbf{Pyk}(\mathbf{Set})$, is defined as the category of sheaves of sets on $\mathbf{Comp}$ with respect to the effective topology:
\[
\mathbf{Pyk}(\mathbf{Set}) := \mathbf{Sh}_{\mathrm{eff}}(\mathbf{Comp}; \mathbf{Set}).
\]
\end{definition}

\begin{definition}[Pyknotic Spaces, {\cite[2.2.3]{pyknotic}}]
The category of \emph{pyknotic spaces}, denoted $\mathbf{Pyk}(\mathcal{S})$, is the category of hypercomplete sheaves \footnote{That is sheaves which satisfy descent for hypercoverings see {{\cite[Definition 6.5.3.2.]{HTT}}} and {{\cite[Corollary 6.5.3.13]{HTT}}} } of spaces on $\mathbf{Comp}$ with respect to the effective topology:
\[
\mathbf{Pyk}(\mathcal{S}) := \mathbf{Sh}_{\mathrm{eff}}^{\mathrm{hyp}}(\mathbf{Comp}; \mathcal{S}).
\]
\end{definition}

\begin{remark}[{\cite[Construction 2.2.12]{pyknotic}}]
The \emph{global sections functor} $\Gamma : \mathbf{Pyk}(\mathcal{S}) \to \mathcal{S}$ is given by evaluating a pyknotic space at the point. For a pyknotic space $X$, we call $\Gamma(X)$ its \emph{underlying space}.
This functor has a left adjoint, the \emph{constant sheaf functor} (as was defined in \ref{link}) which in the context of  $\mathbf{Pyk}(\mathcal{S})$ is denoted  $()^{\operatorname{disc}}: \mathcal{S} \to \mathbf{Pyk}(\mathcal{S})$.
The global sections functor $\Gamma$ also admits a right adjoint in the case of $\mathbf{Pyk}(\mathcal{S})$. 
Namely the indiscrete functor $()^{\operatorname{indisc}} : \mathcal{S} \to \mathbf{Pyk}(\mathcal{S})$
\end{remark}

\begin{remark}
\label{hompre}
From the adjunctions described above, we obtain a pair of geometric morphisms between $\infty$-topoi: the global sections functor $\Gamma : \mathbf{Pyk}(\mathcal{S}) \to \mathcal{S}$ and the indiscrete functor $()^{\operatorname{indisc}} : \mathcal{S} \to \mathbf{Pyk}(\mathcal{S})$. According to {\cite[Remark 6.5.1.4]{HTT}}, the functor $\Gamma$ preserves homotopy groups because it corresponds to the pullback along the geometric morphism $\operatorname{indisc}$. Similarly, the discrete functor $()^{\mathrm{disc}}$, being the pullback along $\Gamma$, also preserves homotopy groups.
\end{remark}

\subsection{Solidification and the Derived Stone Embedding}

In this section, we delve into the concept of solidifying an $\infty$-topos and explore the derived Stone embedding.
Let $\mathcal{C}$ be a small $\infty$-category. The opposite Yoneda embedding $\mathcal{C} \hookrightarrow \operatorname{Fun}(\mathcal{C}, \mathcal{S})^{\mathrm{op}}$ allows us to view objects of $\mathcal{C}$ as representable functor under the opposite Yoneda embedding, we define $\operatorname{Pro}^{\delta_0}(\mathcal{C})$ as the full subcategory of $\operatorname{Fun}(\mathcal{C}, \mathcal{S})^{\mathrm{op}}$ consisting of functors expressible as $\delta_0$-small cofiltred limits of these representable functors (recall that $\delta_0$ denotes the smallest inaccessible cardinal.).

\begin{definition}[Solidification of an $\infty$-topos;  {\cite[Construction 3.3.2]{pyknotic}}]
Let $\mathfrak{X}$ be a bounded $\infty$-topos. Denote by $\mathfrak{X}^{\mathrm{coh}}_{<\infty}$ the full subcategory of truncated coherent objects in $\mathfrak{X}$. The \emph{solidification} of $\mathfrak{X}$ is defined as the category of hypercomplete sheaves on $\operatorname{Pro}^{\delta_0}(\mathfrak{X}^{\mathrm{coh}}_{<\infty})$ with respect to the effective epimorphism topology:
\[
\mathfrak{X}^{\dagger} := \mathbf{Sh}^{\mathrm{hyp}}_{\mathrm{eff}}\big( \operatorname{Pro}^{\delta_0}(\mathfrak{X}^{\mathrm{coh}}_{<\infty}) \big).
\]
\end{definition}

The following theorem shows that when working with an $n$-localic $\infty$-topos, it suffices to consider only the $(n-1)$-truncated coherent objects for the solidification process.

\begin{theorem}[{\cite[Proposition 3.3.9]{pyknotic}}]
Let $n \geq 1$ be an integer, and let $\mathfrak{X}$ be an $n$-localic coherent $\infty$-topos. Then the restriction of presheaves induces an equivalence:
\[
\mathfrak{X}^{\dagger} \simeq \mathbf{Sh}^{\mathrm{hyp}}_{\mathrm{eff}}\big( \operatorname{Pro}(\mathfrak{X}^{\mathrm{coh}}_{\leq n-1}) \big),
\]
with inverse given by right Kan extension.
\end{theorem}
\begin{proof}
Consider the inclusion functor:
\[
 \operatorname{Pro}(\mathfrak{X}^{\mathrm{coh}}_{\leq n-1}) \xhookrightarrow{i} \operatorname{Pro}(\mathfrak{X}^{\mathrm{coh}}_{\leq \infty}) 
\]
This inclusion induces a restriction functor:
\[
\mathfrak{X}^{\dagger} \rightarrow \mathbf{Sh}^{\mathrm{hyp}}_{\mathrm{eff}}\big( \operatorname{Pro}(\mathfrak{X}^{\mathrm{coh}}_{\leq n-1}) \big),
\]
We first check that $i_*$ is fully faithful. This follows directly from Proposition \ref{procov} and the combination of \cite[Proposition 20.4.5.1 and Remark 20.4.5.2]{SAG}.
To establish that $i_*$ is an equivalence, it suffices to show it admits a fully faithful right adjoint $i^!$, as this implies the unit and counit of the adjunction are natural isomorphisms. Define $i^!$ pointwise as the right Kan extension along $i$. Since $i$ is fully faithful, it follows that $i^!$ is also fully faithful.
Finally, we must verify that $i^!$ indeed lands in hypercomplete sheaves. This is a technical verification, fully detailed in \cite[Proposition 3.3.9]{pyknotic}.
\end{proof}
\begin{example}
The $\infty$-category of spaces $\mathcal{S}$ is $1$-localic (see Example~\ref{S}). Therefore, by the theorem above, the solidification of $\mathcal{S}$ is given by:
\[
\mathcal{S}^{\dagger} \simeq \mathbf{Sh}^{\mathrm{hyp}}_{\mathrm{eff}}\big( \operatorname{Pro}(\mathbf{Fin}) \big),
\]
where $\mathbf{Fin}$ denotes the category of finite sets.
\end{example}

Building upon this foundation, \cite{pyknotic} establishes a derived version of Stone duality.

\begin{corollary}[Derived Stone Duality; cf. {\cite[Example 3.3.10]{pyknotic}}]
\label{derivedstone}
There is an equivalence of $\infty$-categories:
\[
\mathbf{Pyk}(\mathcal{S}) \simeq \mathbf{Sh}^{\mathrm{hyp}}_{\mathrm{eff}}\big( \operatorname{Pro}(S_{\pi}) \big),
\]
where $\mathbf{Pyk}(\mathcal{S})$ denotes the category of pyknotic spaces, and $\operatorname{Pro}(S_{\pi})$ is the pro-category of $\pi$-finite spaces.

Moreover, the Yoneda embedding
\[
\operatorname{Pro}(S_{\pi}) \hookrightarrow \operatorname{Fun}\big( \operatorname{Pro}(S_{\pi})^{\mathrm{op}}, \mathcal{S} \big)
\]
lands inside the category of hypercomplete sheaves, inducing a fully faithful functor:
\[
\operatorname{Pro}(S_{\pi}) \hookrightarrow \mathbf{Sh}^{\mathrm{hyp}}_{\mathrm{eff}}\big( \operatorname{Pro}(S_{\pi}) \big),
\]
which embeds the pro-category of $\pi$-finite spaces into $\mathbf{Pyk}(\mathcal{S})$.
\end{corollary}

\begin{proof}
From the previous example, we observe that:
\[
\mathcal{S}^{\mathrm{coh}}_{\leq 0} \simeq \mathbf{Fin}.
\]
Applying the theorem, we obtain:
\[
\mathcal{S}^{\dagger} \simeq \mathbf{Sh}^{\mathrm{hyp}}_{\mathrm{eff}}\big( \operatorname{Pro}(\mathbf{Fin}) \big) \simeq \mathbf{Pyk}(\mathcal{S}).
\]
This establishes the desired equivalence.

Next, note that the effective epimorphism topology on $\operatorname{Pro}(S_{\pi})$ is subcanonical, implying that representable presheaves are indeed sheaves. Since $\pi$-finite spaces are truncated, their images under the Yoneda embedding are truncated and hence hypercomplete.

Finally, because limits of hypercomplete sheaves remain hypercomplete, the pro-objects (being limits of $\pi$-finite spaces) are also hypercomplete. Therefore, the Yoneda embedding induces a fully faithful functor into $\mathbf{Pyk}(\mathcal{S})$.
\end{proof}

\begin{remark}
With this embedding of $\operatorname{Pro}(S_{\pi})$ into $\mathbf{Pyk}(\mathcal{S})$, we will henceforth identify pro-$\pi$-finite spaces with their images in $\mathbf{Pyk}(\mathcal{S})$ without further distinction.
\end{remark}

\begin{remark}
Consider the discrete sheaf functor $\operatorname{disc}: \mathcal{S} \to \mathbf{Pyk}(\mathcal{S})$, which assigns to each space its associated discrete pyknotic space. This functor is fully faithful. The embedding of $\operatorname{Pro}(S_{\pi})$ into $\mathbf{Pyk}(\mathcal{S})$ can be described as the composition:
\[
\operatorname{Pro}(\mathcal{S}) \xrightarrow{\operatorname{Pro}(\operatorname{disc})} \operatorname{Pro}\big( \mathbf{Pyk}(\mathcal{S}) \big) \xrightarrow{\varprojlim} \mathbf{Pyk}(\mathcal{S}),
\]
where $\varprojlim$ denotes the limit in the category of pyknotic spaces.

Therefore, a pro-object $\{X_i\}$ in $\operatorname{Pro}(\mathcal{S})$ can be identified with the limit $\varprojlim \operatorname{disc}(X_i) \in \mathbf{Pyk}(\mathcal{S})$.
\end{remark}
\subsection{Coherent Objects in \texorpdfstring{\ensuremath{\mathbf{Pyk}(\mathcal{S})}}{Pyk(S)}}

We now aim to understand the coherent objects in \(\mathbf{Pyk}(\mathcal{S})\) and \(\mathbf{Pyk}(\text{Set})\). 
\begin{theorem}
\label{pi-finite-coh}
Every pro $\pi$-finite space is a coherent object in the topos $\textbf{Pyk}(\mathcal{S})$.
\end{theorem}
\begin{proof}
By {\cite[Proposition A.2.2.2]{SAG}}, if $X \in \mathbf{Sh}^{\text{hyp}}_{\text{eff}}(\text{Pro}(S_\pi))$ is coherent as an object of $\mathbf{Sh}_{\text{eff}}(\text{Pro}(S_\pi))$, then it remains coherent as an object of $\mathbf{Sh}^{\text{hyp}}_{\text{eff}}(\text{Pro}(S_\pi)) \hookrightarrow \mathbf{Sh}_{\text{eff}}(\text{Pro}(S_\pi))$.

Since the topology on $\text{Pro}(S_\pi)$ is finitary, \ref{A.3.1.3} ensures that the Yoneda embedding lands in coherent objects.
\end{proof}
Next, we recall the complete characterization of the coherent objects in \(\mathbf{Pyk}(\text{Set})\)

\begin{lemma}[{{\cite[2.1.4]{pyknotic}}} or {{\cite[Theorem 2.16 (i)]{condensed2}}}]
Compact Hausdorff spaces are precisely the coherent objects of $\textbf{Pyk}(\textbf{Set})$.
\end{lemma}
\begin{proof}
We have the equivalence:
\[
\textbf{Pyk}(\textbf{Set}) \simeq \mathbf{Sh}_{\text{eff}}(\textbf{Comp}),
\]
where $\textbf{Comp}$ denotes the category of compact Hausdorff spaces. It is well-known that $\textbf{Comp}$ forms a pretopos. 

By {\cite[Proposition C.6.4]{Conceptual}}, the coherent objects of the sheaf category $\mathbf{Sh}_{\text{eff}}(\textbf{Comp})$ correspond to the sheafification of the Yoneda embedding applied to the objects of $\textbf{Comp}$. Therefore, the coherent objects of $\textbf{Pyk}(\textbf{Set})$, denoted $\textbf{Pyk}(\textbf{Set})^{\text{coh}}$, are exactly $\mathcal{Y}(\textbf{Comp})$, where $\mathcal{Y}$ denotes the sheafification of the Yoneda embedding.
\end{proof}
With this in mind, we can now prove that coherent objects are closed under cofiltered limits. This result is not immediately obvious, as coherent objects are generally only closed under finite limits. For instance, in \(\mathbf{Set}\), the coherent objects are the finite sets, which are clearly not closed under cofiltered limits.

\begin{corollary}
\label{cohlim}
Coherent objects in $\textbf{Pyk}(\textbf{Set})$ are closed under cofiltered limits.
\end{corollary}

\begin{proof}
The coherent objects in $\textbf{Pyk}(\textbf{Set})$ are precisely the compact Hausdorff spaces. According to {\cite[Example 2.1.6]{pyknotic}}, there exists a functor from the category of tiny topological spaces $\textbf{Tspc}$ to $\textbf{Pyk(Set)}$, which is a right adjoint and hence preserves limits. Its restriction to $\textbf{Comp}$ is given by the sheafified Yoneda embedding. Thus, we reduce the problem to showing that $\textbf{Comp}$ is closed under cofiltered limits in $\textbf{Tspc}$.

Let $\{ X_i \}$ be a cofiltered diagram in $\textbf{Comp}$. The inverse limit $\varprojlim X_i$ can be identified as a closed subset of the product $\prod X_i$. By Tychonoff’s theorem, the product $\prod X_i$ is compact Hausdorff. Since the inverse limit is a closed subset of a compact Hausdorff space, it is itself a compact Hausdorff space. Consequently, $\varprojlim X_i$ is compact Hausdorff. Therefore, the category $\textbf{Comp}$ is closed under cofiltered limits in $\textbf{Tspc}$.
\end{proof}

\begin{corollary}
\label{consev}
The functor 
\[
(*) : \textbf{Pyk}(\textbf{Set}) \rightarrow \textbf{Set}
\]
is conservative on coherent objects.
\end{corollary}

\begin{proof}
When restricted to the coherent objects, which are precisely the compact Hausdorff spaces, the functor $(*): \textbf{Pyk}(\textbf{Set}) \rightarrow \textbf{Set}$ reduces to the classical forgetful functor from the category of compact Hausdorff spaces to the category of sets.

It is a standard result that this forgetful functor is conservative - a continuous bijection between compact Hausdorff spaces is necessarily a homeomorphism (i.e., an isomorphism in the category of compact Hausdorff spaces). 

Thus, $$(*): \textbf{Pyk}(\textbf{Set}) \rightarrow \textbf{Set}$$ is conservative on coherent objects.
\end{proof}
Finally, we are now ready to understand the homotopy groups of pro-\(\pi\)-finite spaces in \(\mathbf{Pyk}(\mathcal{S})\).
To avoid any ambiguity we will make explicit the discrete embedding in the two following propositions
\begin{proposition}
Let \(\{X_i\}\) be a pro-system of \(\pi\)-finite spaces. Denote by \(p_i: \varprojlim X_i \rightarrow X_i\) the projection maps. Then, we have the following isomorphism:

$$
\pi_j(\varprojlim X_i^{\text{disc}}, x_0)(*) = \varprojlim \pi_j(X_i^{\text{disc}}, p_i(x_0))(*)
$$
\end{proposition}
\begin{proof}

From remark \ref{hompre}
$$\pi_j(\varprojlim X_i^{\text{disc}}, x_0)(*)\simeq \pi_j(\varprojlim (X_i^{\text{disc}} (*)), x_0)\simeq \pi_j(\varprojlim X_i, x_0)$$
  Now applying {\cite[Remark 3.2.8]{DAG-XIII}} we get:
  $$\pi_j(\varprojlim X_i, x_0) \simeq  
\varprojlim \pi_j(X_i, p_i(x_0))$$
Now using again remark \ref{hompre} we get that
$$\pi_j(X_i, p_i(x_0))\simeq \pi_j(X_i^{\text{disc}}, p_i(x_0))$$
\end{proof}
\begin{corollary}
\label{limhom}
Let \(\{X_i\}\) be a pro-system of \(\pi\)-finite spaces. Denote by \(p_i: \varprojlim X_i \rightarrow X_i\) the projection maps. Then, we have the following isomorphism:
\[
\pi_j(\varprojlim X_i^{\text{disc}}, x_0) =\varprojlim \pi_j(X_i^{\text{disc}}, p_i(x_0))
\]
\end{corollary}

\begin{proof}
We know that $\varprojlim X_i^{\text{disc}}$ is coherent, as it is a pro-$\pi$-finite space - \ref{pi-finite-coh}. 

Thus, by Theorem \ref{cohom}, $\pi_j(\varprojlim X_i^{\text{disc}}, x_0)$ is coherent. Similarly, $\varprojlim \pi_j(X_i^{\text{disc}}, p_i(x_0))$ is coherent as an inverse limit of coherent objects in $\textbf{Pyk}(Set)$, Corollary \ref{cohlim}.

From Corollary \ref{consev}, we know that the functor $(*): \textbf{Pyk}(Set) \rightarrow \textbf{Set}$ is conservative on coherent objects. Therefore, it suffices to check that the natural map:
\[
\pi_j(\varprojlim X_i^{\text{disc}}, x_0)(*) \to \varprojlim \pi_j(X_i^{\text{disc}}, p_i(x_0))(*)
\]
is an isomorphism. This is exactly the proposition above.
\end{proof}

\section{Characterization Of The Essential Image}

\subsection{Preparatory Propositions}
We will prove some closure properties that we will need to use in the induction process in the proof of the main theorem
\begin{lemma}
\label{pifiniteq}
The full subcategory $S_\pi\hookrightarrow \textbf{Pyk}(\mathcal{S})$ is closed under geometric realizations of groupoid objects.  
\end{lemma}
\begin{proof}
From the proof of \cite[Corollary A.6.1.7]{SAG}, we know that coherent objects in a topos are closed under geometric realizations of groupoid objects, and the embedding of coherent objects into the topos preserves geometric realizations of groupoid objects.
Since coherent spaces are the \(\pi\)-finite spaces, it follows that \(S_\pi\subset \mathcal{S}\) is closed under geometric realizations of groupoid objects. Furthermore, the embedding \(S_\pi \hookrightarrow \textbf{Pyk}(\mathcal{S})\) factors as \(S_\pi \hookrightarrow \mathcal{S} \hookrightarrow \textbf{Pyk}(\mathcal{S})\). Thus, it remains to verify that the embedding \(\mathcal{S} \to \textbf{Pyk}(\mathcal{S})\) preserves geometric realizations. This is immediate, as the embedding preserves all colimits.
\end{proof}
We will use the notion of group action in an $\infty$-topos as in {\cite[Definition 3.1.]{bundles}}, for an object $X$ with an action of a group object $G$, the quotient is denoted \( X \mathbin{/\mkern-6mu/} G \).
\begin{corollary}
Let \( G \) be a finite group and \( X \) a \(\pi\)-finite space equipped with a \( G \)-action. Then the quotient \( X \mathbin{/\mkern-6mu/} G \), taken in \(\textbf{Pyk}(\mathcal{S})\), is also a \(\pi\)-finite space.
\end{corollary}
\begin{proof}
Remark that the quotient \( X \mathbin{/\mkern-6mu/} G \) can be viewed as the geometric realization of a groupoid object \( (X \mathbin{/\mkern-6mu/} G)_\bullet \) in the category $\operatorname{Grpd}(S_\pi)$ of groupoid objects in $\pi$-finite spaces. Specifically, this groupoid object is defined by:

\begin{itemize}
    \item  \( (X \mathbin{/\mkern-6mu/} G)_0 = X \),
    \item \( (X \mathbin{/\mkern-6mu/} G)_n = X \times G^n \),
\end{itemize}

with face and degeneracy maps encoding the action of $G$ on $X$. 

Since both $X$ and $G$ are $\pi$-finite, each level \( (X \mathbin{/\mkern-6mu/} G)_n \) of the simplicial object is $\pi$-finite. According to Lemma \ref{pifiniteq}, which states that the geometric realization of a groupoid object consisting of $\pi$-finite spaces is itself $\pi$-finite, it follows that the geometric realization \( \left| (X \mathbin{/\mkern-6mu/} G)_\bullet \right| \) is $\pi$-finite.
\end{proof}
\begin{proposition}
The category \( \text{Pro}(S_\pi) \) is closed under pullbacks in \( \textbf{Pyk}(\mathcal{S}) \).
\end{proposition}
\begin{proof}
Since limits commute with limits, we have the following isomorphism for the pullback of pro-objects in \( \textbf{Pyk}(\mathcal{S}) \):
\[
\varprojlim X_j \times_{\varprojlim Y_k} \varprojlim Z_i \simeq \varprojlim (X_j \times_{Y_k} Z_i).
\]
Now, recall that the embedding \( \mathcal{S} \xhookrightarrow{disc} \textbf{Pyk}(\mathcal{S}) \) preserves finite limits. Furthermore, \( S_\pi \) is closed under pullbacks in \( \mathcal{S} \). Therefore, since each \( X_j \times_{Y_k} Z_i \in S_\pi\subset \textbf{Pyk}(\mathcal{S}) \), the inverse limit belongs to \( \text{Pro}(S_\pi) \subset \textbf{Pyk}(\mathcal{S}) \). Thus, we conclude that \( \text{Pro}(S_\pi) \) is closed under pullbacks in \( \textbf{Pyk}(\mathcal{S}) \).
\end{proof}

\begin{lemma}[{\cite[Lemma 2.2.7]{contralimits}}]
\label{climit}
For any object \( X\) in an \( \infty \)-category \( \mathcal{C} \), the canonical forgetful functor \( \pi: \mathcal{C}_{/X} \to \mathcal{C} \) preserves and reflects limits indexed by weakly contractible diagrams.
\end{lemma}
The category \(\text{GAction}(\mathfrak{X})\) refers to the category of objects equipped with a $G$-action as was defined in {\cite[Definition 3.1.]{bundles}}
\begin{lemma}
\label{qprelim}
Let \(\mathfrak{X}\) be an \(\infty\)-topos, and let \(G \in \text{Disc}(\mathfrak{X})\). Then the functor \(\text{GAction}(\mathfrak{X}) \xrightarrow{\mathbin{/\mkern-6mu/} G} \mathfrak{X}\) preserves weakly contractible limits, meaning that \((\varprojlim_{i\in I} P_i) \mathbin{/\mkern-6mu/} G \simeq \varprojlim_{i\in I} (P_i \mathbin{/\mkern-6mu/} G)\), for \(I\) a weakly contractible diagram.
\end{lemma}

\begin{proof}
Since the category \(\text{GAction}(\mathfrak{X})\) is equivalent to \(\mathfrak{X}/\mathbf{B}G\), the statement reduces to lemma \ref{climit}.
\end{proof}

 \begin{proposition}
Let \(\mathfrak{X}\) be an \(\infty\)-topos, and let \( * \xrightarrow{\zeta} X \in \mathfrak{X}_* \) be a connected pointed object. Then the fundamental group \(\pi_1(X, \zeta)\) acts on \(\tilde{X} := X \times_{\tau^{\leq 1} X} *\), and there is an equivalence \(\tilde{X} \mathbin{/\mkern-6mu/} \pi_1(X, \zeta) \simeq X\).
\end{proposition}
\begin{proof}
Under the equivalence between \(\text{GAction}(\mathfrak{X})\) and \(\mathfrak{X}/\mathbf{B}G\) we see that the map \(X \rightarrow \tau^{\leq 1} X \simeq \mathbf{B}G\) induces an action of $G$ on $\tilde{X}$ with the quotient being $X$.  
\end{proof}

\subsection{Main Theorems}
We will now proceed to prove our main theorem. First, we establish a crucial lemma demonstrating that certain relative Eilenberg-MacLane spaces belong to \(\text{Pro}(S_\pi)\). These spaces serve as the bulding blocks for the inductive process used in Theorem \ref{oneside}.

\begin{lemma}
\label{EM-Pr}
Let $\Gamma$ be a $1$-gerb in $\textbf{Pyk}(\mathcal{S})$, let $\zeta: 1 \rightarrow \Gamma$ be a point of $\Gamma$, and let $ K^{/\Gamma}(G, n) $ be an Eilenberg--MacLane object in $\mathbf{Pyk}(\mathcal{S})/\Gamma$. Then, $K^{/\Gamma}(G, n) \in \operatorname{Pro}(S_\pi)$ (viewing $K^{/\Gamma}(G, n)$ as an object of $\mathbf{Pyk}(\mathcal{S})$ via the forgetful functor) if and only if both $\Gamma$ and \( G_0 := G \times_\Gamma \zeta \) lie in $\operatorname{Pro}(S_\pi)$.
\end{lemma}
\begin{proof}
Firsty, from \ref{EM}, we know that $ 1\xrightarrow{\zeta} \Gamma$ is equivalent to $1\rightarrow BH$ for some group object $H$ in $\text{Disc}(\textbf{Pyk}(\mathcal{S}))$(Actually $H$ is $\pi_1(\Gamma, \zeta)$). Since $\Gamma \in Pro(S_\pi)$ from  \ref{otherside} it follows that $H$ is a profinite group.

Next, we consider the pullback of the map $K^{/\Gamma}(G, n) \rightarrow BH$ along $\zeta$:
\[
\begin{tikzcd}
	K \arrow[r] \arrow[d] & {K^{/\Gamma}(G, n)} \arrow[d] \\
	1 \arrow[r, "\zeta"] & BH
\end{tikzcd}
\]
This gives rise to a fiber sequence 
$$ K \rightarrow K^{/\Gamma}(G, n) \rightarrow BH. $$

From {\cite[Theorem 3.17]{bundles}}, this fiber sequence exhibits $K^{/\Gamma}(G, n)$ as the homotopy quotient $K \mathbin{/\mkern-6mu/} H$. By \ref{EM-transfer}, we know that $K \simeq K(G_0, n)$, where $G_0 = G \times_\Gamma \zeta$. Under the functor $\pi_n$ , the action of $H$ on $K(G_0, n)$ gives rise to an action of $H$ on $G_0$ since from \ref{disobj} $\text{Disc}(\mathfrak{X}/K(G_0, n)) \simeq Disc(\mathfrak{X})$ .

From {\cite[Lemma 5.3.3(c)]{Profinite}}, we know that $G_0$ can be written as an inverse limit of finite continuous $H$-modules, i.e., $G_0 \simeq \varprojlim G_0 / N$ where each $N$ is clopen. Applying \ref{limhom}, we obtain $K(G_0, n) \simeq \varprojlim K(G_0/N, n)$.

Next, define $K_N^{/\Gamma} := K(G_0/N, n) \mathbin{/\mkern-6mu/}H$  then from \ref{qprelim}, we deduce that 
$$ K^{/\Gamma}(G, n) \simeq \varprojlim K_N^{/\Gamma}, $$
. Since $G_0 / N$ is finite and $H$ is profinite, the action of $H$ on $G_0/N$ factors through some finite quotient $H/J$. Therefore, $J$ acts trivially on $G_0/N$.

Now, 
$$ K_J^N := K(G_0/N, n) \mathbin{/\mkern-6mu/} J \simeq BJ \times K(G_0/N, n), $$
lies in $Pro(S_\pi)$, as it is a product of objects in $Pro(S_\pi)$ (since $J$ is profinite as a closed subgroup of a profinite group) $$K^N_J:=\varprojlim F_J^N$$ where $F_J^N$ are $\pi$-finite. Now, 
$$ K^{/\Gamma}(G, n) \simeq \varprojlim K_N^{/\Gamma} \simeq \varprojlim (K(G_0/N, n) \mathbin{/\mkern-6mu/} H) \simeq \varprojlim K^N_J \mathbin{/\mkern-6mu/} (H / J) $$ 
Finally, from \ref{qprelim} $K_J^N \mathbin{/\mkern-6mu/} (H / J) \simeq \varprojlim (F_N \mathbin{/\mkern-6mu/} (H / J))$, and by \ref{pifiniteq}, $F_N \mathbin{/\mkern-6mu/} (H / J)$ are $\pi$-finite therefore $K^{/\Gamma}(G, n)\simeq \varprojlim F_N \mathbin{/\mkern-6mu/} (H / J) $ belongs to $ Pro(S_\pi)$.
\end{proof}

We begin by showing that the homotopy groups of pro \(\pi\)-finite spaces are indeed profinite.

\begin{theorem}
\label{otherside}
Let \( X \) be an object in \(\text{Pro}(S_\pi)\). Then for every point \(x_0\) in \( X \), the pyknotic homotopy group \(\pi_j(X, x_0)\) (computed in \(\textbf{Pyk}(\mathcal{S})\)) is profinite for all \( j \).
\end{theorem}

\begin{proof}
This is an immediate consequence of \ref{limhom}.
\end{proof}

Next, we prove a partial converse to Theorem \ref{otherside}.

\begin{theorem}
\label{oneside}
Let \( X \in \textbf{Pyk}(\mathcal{S}) \) be a connected pyknotic space, and let \( * \xrightarrow{\zeta} X \) be a point in \( X \). If \(\pi_j(X, \zeta) \in \text{Disc}(X)\) is profinite for all \( j \), then \( X \) can be expressed as an inverse limit of \(\pi\)-finite spaces. In other words, \( X \in \text{Pro}(S_\pi) \hookrightarrow \textbf{Pyk}(\text{An})\).
\end{theorem}
\begin{proof}
We proceed by induction on the truncation degree $n$ of the space.

\textbf{Base case ($n = 1$)}:  
Consider the pointed object \(\zeta: * \rightarrow X\). Using Proposition~\ref{EM}, which characterizes pointed connected spaces in terms of Eilenberg--MacLane spaces, we have that the map \( * \rightarrow X \) is equivalent to \( * \rightarrow K(\pi_1(X, \zeta), 1) \). Applying Proposition~\ref{limhom}, we conclude that if we can represent \(\pi_1(X, \zeta)\) as an inverse limit \(\varprojlim G_i\) of finite groups \(G_i\), then it follows that
\[
K(\pi_1(X, \zeta), 1) \simeq \varprojlim K(G_i, 1) \in \text{Pro}(\mathcal{S}),
\]
\textbf{Induction step}:  
Assume the statement is true for $n$-truncated spaces. Now, let $X$ be an $(n+1)$-truncated space. Note that $X$ is $1$-connected as an object in $\textbf{Pyk}(\text{An})/\tau^{\leq 1}X$. Therefore, we can apply \ref{pulldia} to obtain the following pullback diagram:
\[\begin{tikzcd}
	{\tau^{\leq n}X} & {\tau^{\leq n-1}X} \\
	{\tau^{\leq 1}X} & {K^{/\tau^{\leq 1}X}(\pi_n(X), n+1)}
	\arrow[from=1-1, to=1-2]
	\arrow[from=1-1, to=2-1]
	\arrow["\lrcorner"{anchor=center, pos=0.125}, draw=none, from=1-1, to=2-2]
	\arrow[from=1-2, to=2-2]
	\arrow[from=2-1, to=2-2]
\end{tikzcd}\]
By \ref{EM-Pr} we see that, $K^{/\tau^{\leq 1}X}(\pi_n(X), n+1) \in \text{Pro}(S_\pi)$. Furthermore, by the induction hypothesis, both $\tau^{\leq n-1}X$ and $\tau^{\leq 1}X$ belong to $\text{Pro}(S_\pi)$. Since $\text{Pro}(S_\pi)$ is closed under pullbacks, it follows that $\tau^{\leq n}X \in \text{Pro}(S_\pi)$.
\\
\textbf{Final step}:  
Finally, since $X$ is the inverse limit of its Postnikov tower and each $\tau^{\leq{n}} X$ can be expressed as an inverse limit of $\pi$-finite spaces, then $X$ is also an inverse limit of $\pi$-finite spaces. 
\end{proof}
Having established the necessary lemmas, we can now deduce our main theorem.

\begin{theorem}[Partial Characterization of the Essential Image]
Consider the embedding \(\text{Pro}(S_\pi) \hookrightarrow \textbf{Pyk}(\mathcal{S})\). A %coherent
connected object in \(\textbf{Pyk}(\mathcal{S})\) lies in the essential image of this embedding if and only if its pyknotic homotopy groups at any chosen basepoint are profinite.
%Moreover, if $X$ is connected we don't need to assume $X$ to be coherent. 
\end{theorem}

\begin{proof}
This follows immediately from Theorems \ref{oneside} and \ref{otherside}.
\end{proof}

\bibliographystyle{alphaurl} % Or 'plainurl', 'abbrvurl', etc.
\bibliography{meta/references} % Name of your .bib file without the extension

\end{document}